\documentclass{amsart}
\usepackage{amscd,amsfonts,amsmath,amssymb,amsthm,curves,enumerate,graphicx,latexsym,hyperref}
\usepackage{bbm}
\usepackage{multirow}
\usepackage{mathrsfs}
\usepackage[mathscr]{eucal}
\usepackage{epsfig,epsf,epic}
\usepackage{epstopdf}

\DeclareMathOperator{\ev}{{ev}}
\DeclareMathOperator{\pt}{{pt}}

\def\bM{\overline M}

\def\bP{\mathbf{P}}
\def\bC{\mathbb{C}}

\def\bF{\mathbb{F}}
\def\bR{\mathbb{R}}
\def\bZ{\mathbb{Z}}

\def\cst{{\bC^*}}

\def\virt{^{\mathrm{vir}}}
\def\lra{\longrightarrow}

\newcommand{\conn}{\nabla}

\newcommand{\pairing}[2]{\left( #1\, , \, #2 \right)}

\newcommand{\integer}{\mathbb{Z}}
\newcommand{\rat}{\mathbb{Q}}
\newcommand{\real}{\mathbb{R}}
\newcommand{\cpx}{\mathbb{C}}

\newcommand{\conste}{\mathbf{e}}

\newcommand{\proj}{\mathbb{P}}

\newtheorem{prop}{Proposition}[section]
\newtheorem{defn}[prop]{Definition}
\newtheorem{lem}[prop]{Lemma}
\newtheorem{thm}[prop]{Theorem}
\newtheorem{cor}[prop]{Corollary}

\newtheorem{eg}[prop]{Example}
\newtheorem{remark}[prop]{Remark}

\begin{document}

\title[Open GW invariants and superpotentials]{Open Gromov-Witten invariants and superpotentials for semi-Fano toric surfaces}
\author[K. Chan]{Kwokwai Chan}
\address{Department of Mathematics\\ The Chinese University of Hong Kong\\ Shatin \\ Hong Kong}
\email{kwchan@math.cuhk.edu.hk}
\author[S.-C. Lau]{Siu-Cheong Lau}
\address{Department of Mathematics\\ Harvard University\\ One Oxford Street\\ Cambridge \\ MA 02138\\ USA}
\email{s.lau@math.harvard.edu}


\begin{abstract}
In this paper, we compute the open Gromov-Witten invariants for every compact toric surface $X$ which is semi-Fano (i.e. the anticanonical line bundle $K_X^{-1}$ is nef). Unlike the Fano case, this involves non-trivial obstructions in the corresponding moduli problem. As a consequence, an explicit formula for the Lagrangian Floer superpotential $W$ is obtained, which in turn gives an explicit presentation of the small quantum cohomology ring of $X$. We also provide a computational verification of the conjectural ring isomorphism between the small quantum cohomology of $X$ and the Jacobian ring of $W$.
\end{abstract}

\maketitle

\tableofcontents

\section{Introduction}

Let $X$ be a compact toric manifold of complex dimension $n$. The mirror for $X$ is given by a so-called {\em Landau-Ginzburg model} which consists of a noncompact complex $n$-dimensional manifold $\check{X}$ together with a holomorphic function $W:\check{X}\to\bC$ called the {\em superpotential}. From the perspective of the Strominger-Yau-Zaslow conjecture \cite{SYZ}, the manifold $\check{X}$, which is a bounded domain in $(\cst)^n$, is given by taking the fiberwise torus dual of the moment map restricted to the complement of toric divisors in $X$ \cite{A,CL}. Furthermore, as shown by the work of Cho-Oh \cite{CO} and Fukaya-Oh-Ohta-Ono \cite{FOOO,FOOO2}, the superpotential $W$ comes from Lagrangian Floer theory for the moment map fibers. More precisely, the coefficients of $W$ are generating functions of genus zero {\em open Gromov-Witten invariants} which are virtual counting of Maslov index two holomorphic stable disks bounded by the Lagrangian torus fibers of the moment map.

In this paper, we investigate the computation of open Gromov-Witten invariants and superpotentials for a class of toric surfaces. Similar problems have been studied by various authors. In \cite{CO}, Cho and Oh classified all non-singular holomorphic disks in a compact toric manifold $X$ with boundary lying in Lagrangian torus fibers. In case $X$ is Fano, since the moduli spaces of holomorphic stable disks do not contain any bubbling configurations, Cho-Oh's results imply that all open Gromov-Witten invariants are equal to one, and hence they obtained an explicit formula for the superpotential $W$, which agrees with the one predicted by Hori and Vafa \cite{HV}. For non-Fano toric manifolds, however, bubbling configurations {\em do} contribute to open Gromov-Witten invariants (Fukaya-Oh-Ohta-Ono \cite{FOOO,FOOO2}), so there are ``quantum corrections" to Hori-Vafa's formula for the superpotential. In this situation, the obstruction theory is non-trivial and this makes explicit computations of open Gromov-Witten invariants much more difficult than the Fano case.

There are very few known computations for non-Fano toric manifolds: In \cite{A2}, by using toric degenerations and studying the wall-crossing phenomenon for disk counting, Auroux was able to compute open Gromov-Witten invariants and write down explicitly the superpotentials for the Hirzebruch surfaces $\bF_2$ and $\bF_3$. Later, Fukaya, Oh, Ohta and Ono \cite{FOOO3}, again making use of toric degenerations, gave another proof for the example $\bF_2$. More recently, the first author \cite{C} established a formula relating open and closed Gromov-Witten invariants. Applying this formula, the open Gromov-Witten invariants for all toric Calabi-Yau surfaces and certain toric Calabi-Yau threefolds (including the total space of the canonical line bundles of any toric Del Pezzo surface) were computed in the joint works \cite{LLW,LLW2} of the second author with Leung and Wu.

The purpose of this paper is to compute all genus zero open Gromov-Witten invariants and hence obtain an explicit formula for the superpotential for any semi-Fano toric surface $X$. We call a compact toric surface {\em semi-Fano} if its anti-canonical bundle is nef (or equivalently, if every toric prime divisor has self-intersection at least $-2$). To state our main result, let $\mathbf{T}$ be a Lagrangian torus fiber of the moment map for a semi-Fano toric surface $X$. Let $b\in\pi_2(X,\mathbf{T})$ be a relative homotopy class of Maslov index two.
\begin{defn}
We call a Maslov index two class $b\in\pi_2(X,\mathbf{T})$ {\em admissible} if and only if $b$ is of the form
$$b=\beta+\sum_{k=-m}^n s_kD_k,$$
where
\begin{enumerate}[(1)]
\item $\beta\in\pi_2(X,\mathbf{T})$ is a class represented by a non-singular holomorphic disk $D^2\subset X$ with boundary $\partial D^2\subset\mathbf{T}$ which intersects a unique irreducible toric divisor $D_0$ with multiplicity one; such a class is called a {\em basic disk class};
\item $D_k$'s are toric prime divisors which form a chain of $(-2)$-curves in $X$;
\item $m,n$ are non-negative integers, and both $s_0\ge s_1\ge s_2\ge\cdots\ge s_n\ge 0$ and $s_0\ge s_{-1}\ge s_{-2}\ge\cdots \ge s_{-m}\ge 0$ are nonincreasing integer sequences with $|s_k-s_{k+1}|=0$ or $1$ for each $k$, and the last term of each sequence is not greater than one.
\end{enumerate}
\end{defn}

We can now state our main result:
\begin{thm}\label{thmGW}
Let $X$ be a compact semi-Fano toric surface. Let $b\in\pi_2(X,\mathbf{T})$ be a relative homotopy class of Maslov index two. Then the genus zero one-pointed open Gromov-Witten invariant $n_b$ is either one or zero according to whether $b$ is admissible or not. As a consequence, the superpotential for the mirror of $X$ is given explicitly by
$$W=\sum_{\substack{b \textrm{ admissible} \\ b \in \pi_2(X,\mathbf{T})}} Z_b,$$
where $Z_b$ is an explicit holomorphic function (in fact a monomial) on $\check{X}$ defined by Equation (\ref{Z_b}).
\end{thm}

The proof of Theorem \ref{thmGW} is based on the comparison between open and closed Gromov-Witten invariants in \cite{C} (and its generalization in \cite{LLW}) and the computation on local Gromov-Witten invariants obtained by Bryan and Leung \cite{BL}. The idea is similar to the proof of Theorem 4.2 in \cite{LLW2}.

As a consequence of Theorem \ref{thmGW}, we can write down an explicit formula for the superpotential for a semi-Fano toric surface (see the tables in Appendix \ref{table}). We apply this to verify the conjectural ring isomorphism between the small quantum cohomology $QH^*(X)$ of a semi-Fano toric surface $X$ and the Jacobian ring $Jac(W)$ of its superpotential $W$ via direct computations.
\begin{cor}\label{thmQC=JAC}
Let $X$ be a compact semi-Fano toric surface, and $W$ its superpotential. Then there is a natural ring isomorphism
\begin{equation} \label{QH=JAC}
QH^*(X)\cong \mathrm{Jac}(W).
\end{equation}
\end{cor}

In a recent work \cite{FOOO4}, Fukaya, Oh, Ohta and Ono proved a much stronger result than Corollary \ref{thmQC=JAC}: For every compact toric manifold $X$ and $\mathbf{b} \in H_*(X)$, there is a ring isomorphism
$$QH^*_{\mathbf{b}} (X) \cong \mathrm{Jac}(W_{\mathbf{b}}) $$
where $QH^*_{\mathbf{b}}(X)$ is the {\em big} quantum cohomology ring and $W_{\mathbf{b}}$ is the superpotential bulk-deformed by $\mathbf{b}$. We remark that their proof uses their big machinery of Lagrangian Floer theory and does not involve explicit computations of open Gromov-Witten invariants.

On the other hand, via the isomorphism (\ref{QH=JAC}), our explicit formula for the supepotential $W$ leads to an explicit presentation of the small quantum cohomology ring $QH^*(X)$ for a semi-Fano toric surface $X$. Indeed we can achieve more:
\begin{cor} \label{bulk_cor}
Let $X$ be a compact semi-Fano toric surface and $\mathbf{b} = D + a X$ be a linear combination of toric cycles, where $D$ is a toric divisor and $a \in \cpx$.  Then the bulk-deformed superpotential is
$$W_{\mathbf{b}} = a + \sum_{\beta \textrm{ admissible}} \exp(\langle \beta, D \rangle) Z_{\beta}.$$
\end{cor}
Then by using the results of Fukaya, Oh, Ohta and Ono mentioned above, an explicit ring presentation of $QH^*_{\mathbf{b}}(X)$ can be obtained for $\mathbf{b}\in H_2(X)\oplus H_4(X)$.

\begin{remark}
Fukaya-Oh-Ohta-Ono \cite{FOOO, FOOO2, FOOO4} used a Novikov ring instead of $\cpx$ as the coefficient ring, which is more appropriate in general.  Throughout this paper we stick to the tradition of using $\cpx$ as the coefficient ring because it turns out that the superpotential $W$, which is a priori a formal power series, is a finite sum for any toric semi-Fano surface $X$. All the statements in this paper remains unchanged if $\cpx$ is replaced by a Novikov ring.
\end{remark}

The rest of this paper is arranged as follows. Section \ref{Review} is a brief review on toric manifolds and their Landau-Ginzburg mirrors. In Section \ref{sec_GW} we compute the open Gromov-Witten invariants for compact semi-Fano toric surfaces and prove Theorem \ref{thmGW}. In Section \ref{qc=Jac}, we outline our computational proof of the isomorphism $QH^*(X)\cong Jac(W)$ and demonstrate the explicit calculations by several examples. Corollary \ref{bulk_cor} is proved in Section \ref{comments}. We end by some further discussions on bulk-deformation by points.

\section*{Acknowledgment}

We are heavily indebted to Conan Leung. He not only suggested this problem to us, but also allowed us to freely use his ideas throughout this paper. We are also grateful to Baosen Wu for numerous inspiring discussions and sharing many of his insights. We would also like to thank Kenji Fukaya, Mark Gross and Yong-Geun Oh for their useful comments on open Gromov-Witten invariants with interior marked points, and an anonymous referee for some very useful comments which help improve the exposition of this paper a lot.

Part of this work was done in the IMS of the Chinese University of Hong Kong, and when the first and second authors were visiting the IH\'{E}S and University of Wisconsin-Madison respectively. The authors would like to thank these institutes for hospitality and providing an excellent research environment. The work of K. Chan described in this paper was substantially supported by a grant from the Research Grants Council of the Hong Kong Special Administrative Region, China (Project No. CUHK404412).

\section{Landau-Ginzburg models as mirrors for toric manifolds}\label{Review}

The purpose of this section is to set up some notations and give a review on certain basic facts in toric geometry and mirror symmetry for toric manifolds that we will need in this paper.

\subsection{A quick review on toric manifolds}
Let $N\cong\integer^n$ be a lattice of rank $n$. For simplicity we will always use the notation $N_R:=N\otimes R$ for a $\integer$-module $R$. Let $X_\Sigma$ be a compact complex toric $n$-fold $X_\Sigma$ defined by a fan $\Sigma$ supported in $N_\real$. $X_\Sigma$ admits an action by the complex torus $N_\cpx/N\cong(\cpx^\times)^n$, whence its name ``toric manifold". There is an open orbit in $X_\Sigma$ on which $N_\cpx/N$ acts freely, and by abuse of notation we shall also denote this orbit by $N_\cpx/N\subset X_\Sigma$.

We denote by $M$ the dual lattice of $N$. Every lattice point $\nu \in M$ gives a nowhere-zero holomorphic function $\exp\pairing{\nu}{\cdot}:N_\cpx/N\to\cpx$ which extends to a meromorphic function on $X_\Sigma$. Its zero and pole sets define a toric divisor which is linearly equivalent to the zero divisor. (By a toric divisor $X_\Sigma$ we mean a divisor $D\subset X$ which is invariant under the action of $N_\cpx/N$.)

If we further equip $X_\Sigma$ with a toric K\"ahler form $\omega$, then the action of $N_\real/N$ on $X_\Sigma$ induces a moment map
$$\mu_0:X_{\Sigma}\to M_\real,$$
whose image is a polytope $P\subset M_\real$ defined by a system of inequalities
$$\pairing{v_i}{\cdot}\geq c_i,\ i=1,\ldots,d,$$
where $v_i$ are all primitive generators of rays of $\Sigma$, and $c_i\in\real$ are some suitable constants.

$P$ admits a natural stratification by its faces. Each codimension-one face $T_i\subset P$ which is normal to $v_i\in N$ gives an irreducible toric divisor $D_i=\mu_0^{-1}(T_i)\subset X_\Sigma$ for $i=1,\ldots,d$, and all other toric divisors are generated by $\{D_i\}_{i=1}^d$.  For example, the anti-canonical divisor of $X_\Sigma$ is given by $\sum_{i=1}^d D_i$.

\subsection{Gromov-Witten invariants}
First we recall the definition of closed Gromov-Witten invariants for a projective manifold. Let $\beta\in H_2(X,\bZ)$ be a $2$-cycle in a smooth projective variety $X$. Let $\bM_{g,k}(X,\beta)$ be the moduli space of stable maps
$$f:(C;x_1,\cdots x_k)\lra X,$$
where $C$ is a genus $g$ nodal curve with $k$ marked points and $f_*[C]=\beta$. Let $\ev_i:\bM_{g,k}(X,\beta)\to X$ ($i=1,\ldots,k$) be the evaluation maps $f\mapsto f(x_i)$.
\begin{defn}
Given cohomology classes $\gamma_i\in H^*(X)$, $1\le i\le k$, the closed Gromov-Witten invariant $GW^{X,\beta}_{g,k}(\gamma_1,\cdots,\gamma_k)$ is defined by
\[GW^{X,\beta}_{g,k}(\gamma_1,\cdots,\gamma_k):=\int_{[\bM_{g,k}(X,\beta)]^{\virt}}
\prod_{i=1}^k\ev_i^*(\gamma_i),\]
where $[\bM_{g,k}(X,\beta)]^{\virt}$ denotes the virtual fundamental class of the moduli space $\bM_{g,k}(X,\beta)$.
\end{defn}


For toric manifolds, Fukaya-Oh-Ohta-Ono \cite{FOOO} defined open Gromov-Witten invariants as follows. Let $X=X_\Sigma$ be a toric manifold defined by a fan $\Sigma$. For a moment map Lagrangian torus fiber $\mathbf{T}\subset X$, let $\pi_2(X,\mathbf{T})$ be the group of homotopy classes of maps
$$u:(\Delta,\partial\Delta)\lra(X,\mathbf{T})$$
where $\Delta:=\{z\in\cpx:|z|\leq1\}$ denotes the standard closed unit disk in $\cpx$. Then $\pi_2(X,\mathbf{T})$ is generated by the {\em basic disk classes} $\beta_i\in\pi_2(X,\mathbf{T})$ which correspond to the primitive generators $v_i\in N$ of rays in $\Sigma$ for $i=1,\ldots,d$. The two most important classical symplectic invariants associated to $\beta\in\pi_2(X,\mathbf{T})$ are its symplectic area $\int_\beta\omega$ and its Maslov index $\mu(\beta)$.

Now for $\beta\in\pi_2(X,\mathbf{T})$, let $\bM_k(\mathbf{T},\beta)$ be the moduli space of stable maps from a bordered Riemann surface of genus zero with $k$ boundary marked points respecting the cyclic order of the boundary in the class $\beta$. Notice that the bordered Riemann surface could have disk or sphere bubbles. It is known that $\bM_k(\mathbf{T},\beta)$ has expected dimension $n+\mu(\beta)+k-3$. Let $[\bM_k(\mathbf{T},\beta)]^{\virt}$ be its virtual fundamental chain constructed in \cite{FOOO}. We let
$$\mathrm{ev}_i:\bM_k(\mathbf{T},\beta)\lra\mathbf{T}$$
be the evaluation maps defined by $\mathrm{ev}_i([u;p_0,\ldots,p_{k-1}])=u(p_i)$ for $0\le i\le k-1$.

Consider the case $k=1$ and $\mu(\beta)=2$. Note that the virtual dimension of $\bM_1(\mathbf{T},\beta)$ is equal to $\dim\mathbf{T}=n$ if and only if $\mu(\beta)=2$. Since the minimal Maslov index is two, the virtual fundamental chain $[\bM_k(\mathbf{T},\beta)]^{\virt}$ becomes a cycle when $\mu(\beta)=2$. Hence we can define:
\begin{defn}[\cite{FOOO}] \label{one-pt_openGW}
Given a Lagrangian torus fiber $\mathbf{T}\subset X$ and $\beta\in\pi_2(X,\mathbf{T})$, the genus zero one-pointed open Gromov-Witten invariant $n_\beta$ is defined as
$$n_\beta:=\mathrm{ev}_{0*}([\bM_1(\mathbf{T}, \beta)]^{\virt})\in H_n(\mathbf{T};\rat)\cong\rat.$$
\end{defn}
It was shown in \cite{FOOO} that the number $n_\beta$ is independent of the perturbations used to define the virtual fundamental cycle and hence the above indeed defines an invariant. One should view the invariant $n_\beta\in\rat$ as the virtual number of holomorphic stable disks representing the class $\beta$ such that their boundaries pass through a fixed generic point in $\mathbf{T}$.

Let us consider the situation where $X=X_\Sigma$ is semi-Fano, i.e. with nef anti-canonical line bundle. By the classification result of Cho-Oh \cite{CO}, a class $\beta\in\pi_2(X,\mathbf{T})$ represented by a stable disk must be of the form $\beta=\beta'+\alpha$, where $\beta'$ is a disk class represented by sum of holomorphic disks and $\alpha\in H_2(X)$ is represented by sum of rational curves. The Maslov index of $\beta'$ is $2k$ where $k$ is the intersection number of $\beta'$ with the toric anti-canonical divisor, and the first Chern number $c_1(\alpha):=\int_\alpha c_1(X)$ of $\alpha$ must be non-negative since $X$ is semi-Fano. This shows that any holomorphic stable disk with Maslov index two must be of the form $\beta_i+\alpha$ where $\beta_i$ is a basic disk class and $\alpha\in H_2(X)$ is an effective curve class with first Chern number $c_1(\alpha)=0$.

\subsection{The LG mirror of toric manifolds}
The mirror of a toric manifold $X= X_{\Sigma}$ is a Landau-Ginzburg model $(\check{X},W)$, which consists of a noncompact complex manifold $\check{X}$ together with a holomorphic function $W:\check{X}\to\bC$ called the superpotential. From the perspective of Lagrangian Floer theory, the superpotential $W$ comes from the boundary-deformed Floer potential for Lagrangian torus fibers, and can be written down in terms of K\"ahler parameters and open Gromov-Witten invariants of $X$ \cite{CO,FOOO,FOOO2,A}. The following is a brief review of this procedure from the SYZ viewpoint. See \cite{CL} for more details.

First of all, we recall that the {\em semi-flat mirror} of $X$ is
$$\check{X}_0:=\big\{(\mathbf{T}_r,\conn):r\in P^{\mathrm{int}},\conn\textrm{ is a flat $U(1)$-connection on $\mathbf{T}_r$}\big\},$$
where $\mathbf{T}_r\subset X$ denotes the moment-map fiber over $r$ and $P^{\mathrm{int}}$ denotes the interior of $P$. It is well known that $\check{X}_0$ can be equipped with the so-called semi-flat complex structure, making it into a complex manifold \cite{L}. In this toric case, $\check{X}_0$ is simply $P^{\mathrm{int}}\times M_\real/M$ equipped with the standard complex structure.

Let $\Lambda^*$ be the lattice bundle over $B_0$ whose fiber at $r\in P^{\mathrm{int}}$ is $\Lambda^*_r=\pi_1(\mathbf{T}_r)$.  For each $\lambda\in\Lambda^*$, we may consider the following weighted count of stable holomorphic disks:
$$\mathcal{F}(\lambda):=\sum_{\partial\beta=\lambda}n_\beta\exp\left(-\int_\beta\omega\right).$$
This defines a function $\mathcal{F}:\Lambda^*\to\real$. Applying fiberwise Fourier transform on $\mathcal{F}$, we obtain the superpotential
\begin{align*}
W:\check{X}_0&\to\cpx, \\
W(\mathbf{T}_r,\conn)&=\sum_{\beta\in\pi_2(X,\mathbf{T}_r)}n_\beta\exp\left(-\int_\beta\omega\right)
\textrm{Hol}_\conn(\partial\beta).
\end{align*}
Notice that the above expression can be an infinite series. Nevertheless we will see that for semi-Fano toric surfaces, this is just a finite sum and hence there are no convergence issues for all the examples we consider in this paper. In general, assuming convergence, $W$ is a holomorphic function on $\check{X}_0$.

For $\beta\in\pi_2(X,\mathbf{T}_r)$, we define a function $Z_\beta:\check{X}_0\to\cpx$ by
\begin{equation} \label{Z_b}
Z_\beta(\mathbf{T}_r,\conn):=\exp\left(-\int_\beta\omega\right)\textrm{Hol}_\conn(\partial\beta),
\end{equation}
so that the superpotential can be written in the form $W=\sum_{\beta\in\pi_2(X,T_r)}n_\beta Z_\beta$. Note that $Z_\beta$ is holomorphic and in fact it is a monomial in terms of the standard coordinates on $M_{\bC^*}$.

It is already known by Cho-Oh \cite{CO} that $n_{\beta_i}=1$ for all the basic disk classes $\beta_i$. When $X$ is semi-Fano, as we have seen above, the moduli space $\bM_1(\mathbf{T},\beta)$ is non-empty only when $\beta=\beta_i+\alpha$ for some $i=1,\ldots,d$ and $\alpha\in H_2(X)$ represented by a rational curve of Chern number zero. Thus we may write
$$W=W_0+\sum_{i=1}^d\sum_{\alpha\neq0,c_1(\alpha)=0}n_{\beta_i+\alpha}Z_{\beta_i+\alpha},$$
where $W_0=\sum_{i=1}^d Z_{\beta_i}$. In general it is very hard to compute $n_{\beta_i+\alpha}$ starting from the definition. In the following section, we will give a method to compute these invariants when $X$ is a semi-Fano toric surface.


\section{Disk counting and GW invariants}\label{sec_GW}

\subsection{A fact on toric surfaces}

In this subsection we discuss some elementary results on toric surfaces, which will be needed in the proof of Theorem \ref{thmGW}. These are probably well-known facts among experts; but for convenience of the reader, we include their proofs here.

We start with the well-known formula for the self-intersection number of a toric prime divisor in a compact toric surface. Let $X=X_\Sigma$ be a smooth toric surface defined by a fan $\Sigma$ in $\bZ^2$. Suppose $D\subset X$ is a compact toric prime divisor.  Then $D$ corresponds to a ray $\tau\in\Sigma$, so that $\tau=\sigma^-\cap\sigma^+$ for two $2$-dimensional cones $\sigma^-,\sigma^+\in\Sigma$. (See Figure \ref{cone}).

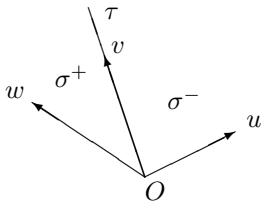
\begin{figure}[h!]
\begin{center}
\setlength{\unitlength}{3mm}
\begin{picture}(30,7)(-15,0)
\linethickness{0.075mm}
\put(0,0){\vector(-1,3){1.8}}\put(0,0){\line(-1,3){2.5}}
\put(0,0){\vector(-3,2){5}}
\put(0,0){\vector(2,1){4}}
\put(4.5,2.2){$u$}
\put(-1.5,5.5){$v$}
\put(-6.2,3.5){$w$}
\put(-1.8,7){$\tau$}
\put(-4,4){$\sigma^+$}
\put(1,3){$\sigma^-$}
\put(0,-1){$O$}
\end{picture}
\end{center}
\caption{Cones corresponding to a compact divisor.}\label{cone}
\end{figure}

Let $\tau$ be generated by $v\in \bZ^2$, $\sigma^-$ be generated by $u,v$ and $\sigma^+$ be generated by $v,w$ such that $u,v,w$ are placed in a counterclockwise fashion. Then the self-intersection of $D$ is given by
\[ D^2=-\left|
\begin{matrix}
u_1 & w_1\\
u_2 & w_2
\end{matrix}
\right|\]
where
\[ u=\binom{u_1}{u_2}\quad\text{and}\quad w=\binom{w_1}{w_2}.\]

\begin{prop}\label{minus}
Let $D=\cup_{i=1}^l D_i$ be a connected union of compact toric prime divisors with $D_i^2=-2$, and $\tau_i$ be the ray corresponding to $D_i$. Suppose $\sigma_i\in\Sigma$ are $2$-dimensional cones so that $\tau_i=\sigma_{i-1}\cap\sigma_i$. Then the cone $\cup_{i=0}^n\sigma_i$ is strictly convex.
\end{prop}

\begin{proof}
Suppose $\tau_i$ is generated by $v_i\in\bZ^2$. Without loss of generality, we can assume $v_i$ are labeled in a counterclockwise order as vectors in $\bR^2$. We further let $\sigma_0$ be generated by $v_0,v_1$; and $\sigma_n$ be generated by $v_n,v_{n+1}$.

Let
\[v_i=\binom{a_i}{b_i}.\]
Since $D_i$ is a $(-2)$-curve, we have
\[\left|
\begin{matrix}
a_{i-1} & a_{i+1}\\
b_{i-1} & b_{i+1}
\end{matrix}\right|=2.\]
In other words, the area of the triangle spanned by $v_{i-1}$ and $v_{i+1}$ is $1$.

\begin{figure}[h!]
\begin{center}
\setlength{\unitlength}{3mm}
\begin{picture}(30,7)(-15,0)
\linethickness{0.075mm}
\put(0,0){\vector(-1,3){1.6}}
\put(0,0){\vector(-3,2){4.8}}
\put(0,0){\vector(1,4){1.6}}
\put(-2.5,5.5){$v_{k}$}
\put(-6.5,4){$v_{k+1}$}
\put(2,6){$v_{k-1}$}
\put(-2.8,2.3){$B$}
\put(-0.5,4){$A$}
\put(-8.5,1.5){$L$}
\put(3.8,7.5){\line(-2,-1){11}}
\put(0,-1){$O$}
\end{picture}
\end{center}
\caption{$-2$ toric divisors.}\label{triangle}
\end{figure}
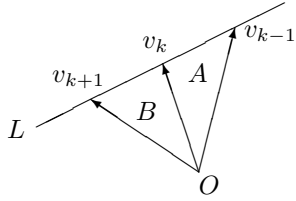

On the other hand, let $A$ be the triangle spanned by vectors $v_{i-1}$ and $v_i$; and let $B$ be the triangle spanned by $v_i$ and $v_{i+1}$. Since $X$ is smooth, the areas of $A$ and $B$ are $\frac{1}{2}$. Now because the sum of areas of $A$ and $B$ is $1$, which is equal to the area of the triangle spanned by $v_{i-1}$ and $v_{i+1}$, we know the heads of the vectors $v_{i-1}$, $v_i$ and $v_{i+1}$ are on the same line $L$. Moreover,
\[v_i=\frac{1}{2}(v_{i-1}+v_{i+1}).\]

Now since the heads of all vectors $v_i$ are on the same line, the cone $\cup_{i=0}^n\sigma_i$ must be strictly convex.
\end{proof}

\subsection{Proof of Theorem \ref{thmGW}}

We are now in a position to give a proof of the main result (Theorem \ref{thmGW}) of this paper.

Let $X$ be a compact semi-Fano toric surface. Let $D_1,\cdots, D_d$ denote the toric prime divisors of $X$. Let $\mathbf{T}$ be a Lagrangian torus fiber and let $\beta_i\in\pi_2(X,\mathbf{T})$ be the basic disk class such that $\beta_i\cdot D_j=\delta_{ij}$. Recall that, given any $b\in\pi_2(X,\mathbf{T})$ of Maslov index two, the moduli space $\bM_{1}(\mathbf{T},b)$ of stable maps from bordered Riemann surfaces of genus zero with one boundary marked point to $X$ in the class $b$ is empty unless $b=\beta_i$, or $b=\beta_i+\alpha$ for some $i\in\{1,\ldots,d\}$ and $\alpha\in H_2(X,\bZ)$ with $c_1(\alpha)=0$. Moreover, such an $\alpha$ must be of the form $\alpha=\sum s_kD_k$ where all $D_k$ have self-intersection $-2$.

Our goal is to compute the open Gromov-Witten invariant $n_b$ for all Maslov index two classes $b\in\pi_2(X,\mathbf{T})$. To state the result, we need the following definitions.
\begin{defn}\label{def_seqad}
Let $m_1, m_2\in\bZ$. We call a sequence $\{s_k \mid m_1\le k\le m_2\}$ admissible with center $l$ if each $s_k$ is a positive integer, and
\begin{enumerate}
\item $s_i\leq s_{i+1}\leq s_i+1$ when $i<l$;
\item $s_i\ge s_{i+1}\ge s_i-1$ when $i\ge l$;
\item $s_{m_1}, s_{m_2}\leq1$.
\end{enumerate}
\end{defn}

For any toric prime divisor $D_i$ with self-intersection $-2$, we have a maximal chain $D_i^{\max}$ of compact toric $(-2)$-divisors containing $D_i$. Given a sequence $\{s_k\}$, we have an induced sequence $\{\tilde s_k\}$ with respect to $D_i$, defined as $\tilde s_j=s_j$ if $D_j\subset D_i^{\max}$ and $s_j=0$ otherwise.

\begin{defn}\label{def_ad}
Let $b=\beta_i+\alpha$ with $\alpha=\sum s_kD_k$. We say $b$ is admissible if $D_i^2=-2$ and the sequence $\{s_k\}$ is identical to its induced sequence with respect to $D_i$, and $\{s_k\}$ is admissible with center $i$.
\end{defn}

To prove Theorem \ref{thmGW}, we recall the computations of local Gromov-Witten invariants for a configuration of $\bP^1$'s in a complex surface which was obtained by Bryan and Leung in \cite{BL} as follows. Let $L(n)$ be a genus $0$ nodal curve consisting of a linear chain of $2n+1$ smooth components $L_{-n},\cdots, L_n$ with an additional smooth component $L_*$ meeting $L_0$. So we have $L_k\cap L_j=\emptyset$ unless $|k-j|=1$ and $L_*\cap L_k=\emptyset$ unless $k=0$. It was shown in \cite{BL} that $L(n)$ can be embedded into a smooth surface $S$ so that all $L_k$ are $(-2)$-curves and $L_*$ is a $(-1)$-curve, where $S$ can be taken as a certain blowup of $\bP^2$ along points.

\begin{figure}[h!]
\begin{center}
\setlength{\unitlength}{3mm}
\begin{picture}(30,5)(-15,-1)
\linethickness{0.075mm}
\put(0,0){\line(0,1){3}}\put(0,3){\circle*{0.2}}
\put(-15,0){\circle*{0.2}}
\multiput(-6,0)(3,0){4}{\line(1,0){3}\circle*{0.2}}
\put(-15,0){\line(1,0){3}\circle*{0.2}}
\put(-6,0){\circle*{0.2}}
\put(-15,0){\circle*{0.2}}\put(-12,0){\line(1,0){0.8}}
\put(-6,0){\line(-1,0){0.8}}\put(-10,-0.3){$\cdots$}
\put(12,0){\circle*{0.2}}\put(12,0){\line(1,0){3}\circle*{0.2}}
\put(12,0){\line(-1,0){0.8}}
\put(6,0){\line(1,0){0.8}}\put(8,-0.3){$\cdots$}
\put(-0.2,3.7){$1$}
\put(-0.3,-1){$s_0$}
\put(2.7,-1){$s_1$}
\put(5.7,-1){$s_2$}
\put(11.7,-1){$s_{n-1}$}
\put(14.7,-1){$s_n$}
\put(-3.3,-1){$s_{-1}$}
\put(-6.3,-1){$s_{-2}$}
\put(-15.3,-1){$s_{-n}$}
\put(-12.3,-1){$s_{-n+1}$}
\end{picture}
\end{center}
\caption{The graph of $L(n)$.}\label{tree}
\end{figure}
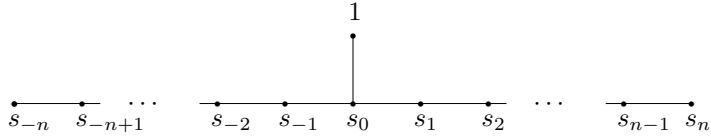

The local Gromov-Witten invariant of $L(n)$ is well-defined\footnote{The local Gromov-Witten invariants of $L(n)$ are the usual genus-zero Gromov-Witten invariants with no marked point of the curve classes $L_*+\sum_{k} s_kL_k$ in $S$.  It turns out that the invariants are independent of the choice of the surface $S$ that $L(n)$ is embedded into, and this is the meaning of well-defindness in this context.} for the curve classes
\[L_*+\sum_{k=-n}^{n} s_kL_k,\quad s_k\ge 0.\]
\begin{thm}\cite{BL}\label{thmBL}
{The genus zero local Gromov-Witten invariants $N(s_k)$ of $L(n)$ for the class $L_*+\sum_{k=-n}^{n} s_kL_k$ is given by
\[ N(s_k)=\left\{ \begin{matrix} 1 & \text{if $\{s_k\}$ is admissible with center $0$.}\\
0& \text{otherwise.}\end{matrix}\right.\]
}\end{thm}

We remark that admissible with center $0$ here is an equivalent term for \textit{$1$-admissible} used in \cite{BL}, and the invariant $N(s_k)$ is independent of the surface $S$ into which we embed $L(n)$.

We will need to apply Theorem \ref{thmBL} to situations which are (apparently) more general than those we described above. Let us explain why this can be done as follows. Let $\mathbf{s}:=\{s_k \mid m_1\le k\le m_2\}$ (with $m_1\leq0$ and $m_2\geq0$) be a sequence which is admissible with center $0$. Let $L(\mathbf{s})$ be a genus $0$ nodal curve consisting of a linear chain of $m_2-m_1+1$ smooth components $L_{m_1},\ldots,L_{m_2}$ with an additional smooth component $L_*$ meeting $L_0$, so that we have $L_k\cap L_j=\emptyset$ unless $|j-k|=1$ and $L_*\cap L_k=\emptyset$ unless $k=0$. We shall regard $L(\mathbf{s})$ as part of the genus 0 nodal curve $L(n)$ where $n=\mathrm{max}\{-m_1,m_2\}$. Now let $\overline M(\mathbf{s})$ be the moduli space of genus 0 stable maps to $L(\mathbf{s})$ in the class
$$L_*+\sum_{k=m_1}^{m_2} s_kL_k.$$
Also let $\overline M(n)$ be the moduli space of genus 0 stable maps to $L(n)$ in the class
$$L_*+\sum_{k=-n}^{n} s_kL_k,$$
where we set $s_k=0$ for $k<m_1$ and for $k>m_2$. Then the moduli spaces $\overline M(\mathbf{s})$ and $\overline M(n)$ can naturally be identified. More importantly, the proof of Lemma 5.3 on p. 385 of \cite{BL} shows that they have the same deformation and obstruction theories, and hence define the same invariants. Therefore, Theorem \ref{thmBL} still works in situations where the configuration of curves $L(\mathbf{s})$ is not symmetric with respect to the center $L_0$.

\begin{proof}[Proof of Theorem \ref{thmGW}]
Given a semi-Fano toric surface $X$ defined by a fan $\Sigma$, we would like to compute the open Gromov-Witten invariant $n_b$ for $b\in\pi_2(X,\mathbf{T})$. First of all, by \cite{CO, FOOO}, $n_b$ is non-zero only when $b=\beta_i+\alpha$ for some $i$ and $\alpha\in H_2(X,\bZ)$ represented by rational curves with $c_1(\alpha)=0$. It is already known that $n_b = 1$ when $\alpha = 0$, so it suffices to consider $\alpha \not= 0$.

Suppose $n_{\beta_i+\alpha} \neq 0$ and $\alpha \neq 0$. Then $D_i$ must have self-intersection $-2$, and $\alpha$ must be of the form $\alpha = \sum_{k \in I} s_k D_k$, where $I$ is the index set containing all the natural numbers $k$ such that $D_k \subset D_i^{\mathrm{max}}$, and $s_i \not= 0$.  We want to show that the sequence $\{s_k\}$ is admissible, and in such cases $n_b = 1$.

This is done by equating the open Gromov-Witten invariant $n_b$ to a closed Gromov-Witten invariant of yet another toric manifold $Y$, which is a toric modification of $X$.  The modification is constructed as follows. Let $v_i$ be the primitive generator of the ray of $\Sigma$ corresponding to $D_i$, and let $\Sigma_1$ be the refinement of $\Sigma$ by adding the ray generated by $v_\infty := -v_i$ (and then completing it to a complete fan)\footnote{Note that $\Sigma_1$ may no longer be convex, which means that $X_{\Sigma_1}$ may no longer be semi-Fano.  Nevertheless, as we shall see in the next paragraph, that the curve classes under our consideration never intersect the new toric divisor and they are rigid (because they are $(-2)$ curves).  Thus their invariants can be computed by Theorem \ref{thmBL}.}.  In general the corresponding toric variety $X_{\Sigma_1}$ may not be smooth. If this is the case, then we take a toric desingularization $Y$ of $X_{\Sigma_1}$ by adding rays which are adjacent to $v_\infty$.  By abuse of notations we still denote the divisors in $Y$ corresponding to $v_l$'s by $D_l$, and $\alpha = \sum_{k \in I} s_k D_k$ is regarded as a homology class in $Y$. We remark that the above procedure does nothing if the ray generated by $v_\infty$ is already in $\Sigma$.

By Proposition \ref{minus}, since $D_k$'s have self-intersection $(-2)$ for $k \in I$ ($I$ is the index set introduced above) and $\cup_{k\in I} D_k$ is connected, the union of the top-dimensional cones adjacent to the rays generated by $v_k$ for $k \in I$ is strictly convex.  Thus the ray generated by $v_\infty$ cannot be adjacent to those generated by $v_k$ for $k \in I$.  Then the newly added rays are not adjacent to any $v_k$ for $k \in I$, and thus each $D_k \subset Y$ for $k \in I$ still has self-intersection number $(-2)$.  Let $f \in H_2 (Y)$ be the fiber class, that is, $f = \beta_i + \beta_\infty$, where $\beta_\infty$ is the disk class corresponding to $v_\infty$.

By Theorem 1.1 in \cite{C} and its generalization in \cite{LLW}, we have the following equality between open and closed Gromov-Witten invariants:
$$n_b=GW^{Y,\alpha+f}_{0,1}([\pt]).$$
The proof was by showing that the open moduli space $\overline M_1^{\mathrm{ev}=p}(X,b)$ and the closed moduli space $\overline M_1^{\mathrm{ev}=p}(Y,f + \alpha)$ are isomorphic as Kuranishi spaces.  We refer the reader to \cite{C, LLW} for details.

Next we identify $GW^{Y,\alpha+f}_{0,1}([\pt])$ with the local Gromov-Witten invariant of a configuration of $\bP^1$'s. Let $\tilde Y$ be the blowup of $Y$ at a generic point $p$. Then, by the result of Hu \cite{H} and Gathmann \cite{G}, which relates Gromov-Witten invariants of blowups along points, we know that the Gromov-Witten invariant of $Y$ for a class $\gamma$ with one point constraint is equal to that of $\tilde Y$ for the strict transform of $\gamma$ without this point constraint. More precisely, we have
$$GW^{Y,\alpha+f}_{0,1}([\pt])=GW^{\tilde Y,\alpha+f'}_{0,0},$$
where $f'$ is the strict transform $f$, which is the class of a $(-1)$-curve.

Because $\alpha=\sum s_kD_k$, with all $D_k$ have self-intersection $-2$, it is easy to see that every curve in $\alpha+f'$ is a tree of $\bP^1$'s, with the same configuration as $L(\mathbf{s})$, up to a relabeling of its indices and shifting of the center. Therefore, $GW^{\tilde Y,\alpha+f'}_{0,0}$ is precisely the local Gromov-Witten invariant of $L(\mathbf{s})$ for the class $f'+\alpha$. Theorem \ref{thmGW} now follows from Theorem \ref{thmBL} and the discussion that follows.
\end{proof}

Theorem \ref{thmGW} allows us to explicitly compute the superpotential for any compact semi-Fano toric surface. Since these surfaces can be completely classified (there are totally 16 such surfaces, 5 of which are Fano), we can give explicit formulas for all their superpotentials; a list of which is given in the appendix. In a very recent work Fukaya-Oh-Ohta-Ono \cite{FOOO5}, our explicit formula for the superpotential $W$ of the semi-Fano toric surface $X_{11}$ in the table was used in their proof of the existence of a continuum of mutually disjoint non-displaceable Lagrangian tori in a cubic surface.

\section{Small quantum cohomology and Jacobian ring}\label{qc=Jac}

For a toric Fano manifold $X$, the map
$$\psi:QH^*(X)\rightarrow Jac(W),\ D_i\mapsto Z_{\beta_i},$$
gives a canonical ring isomorphism between the small quantum cohomology $QH^*(X)$ of $X$ and the Jacobian ring $Jac(W)$ of the superpotential $W$ \cite{CL,FOOO}. Recall that the Jacobian ring of $W$ is defined as
$$Jac(W)=\bC[z_1^{\pm1},\ldots,z_n^{\pm1}]/\langle\partial_1W,\ldots,\partial_nW\rangle,$$
where $\partial_j$ denotes $z_j\frac{\partial}{\partial z_j}$ and $n=\dim X$. In the non-Fano case, it is expected that we still have an isomorphism $QH^*(X)\cong Jac(W)$,\footnote{This is now proved in the recent work \cite{FOOO4} of Fukaya, Oh, Ohta and Ono (as a special case of their main result).} but the map $\psi:QH^*(X)\rightarrow Jac(W)$ needs to be modified by quantum corrections.

In the following, we briefly recall the definition of the corrected map following Fukaya, Oh, Ohta and Ono \cite{FOOO,FOOO2}. As before, $X$ is a compact toric manifold and $\mathbf{T}$ is a Lagrangian torus fiber. Consider the moduli space $\bM_{k,l}(\mathbf{T},\beta)$ of stable maps from genus 0 bordered Riemann surfaces to $(X,L)$ with $k$ boundary marked points and $l$ interior marked point in the class $\beta$. We have
evaluation maps
$$\mathrm{ev}^\mathrm{int}:\bM_{k,l}(\mathbf{T},\beta)\rightarrow X^l,\ [u;p_0,p_1,\ldots,p_{k-1};z_1, \ldots, z_l]\mapsto (u(z_1), \ldots, u(z_l)),$$
and
$$\mathrm{ev}_i:\bM_{k,1}(\mathbf{T},\beta)\rightarrow\mathbf{T},\ [u;p_0,p_1,\ldots,p_{k-1};z]\mapsto u(p_i),$$
$i=0,1,\ldots,k-1$, at the interior and boundary marked points respectively.

Let $V_1, \ldots, V_l \subset X$ be toric subvarieties. Consider the fiber product
$$\bM_{1,l}(\mathbf{T},\beta;V_1, \ldots, V_l):=\bM_{1,l}(\mathbf{T},\beta)_{\mathrm{ev}^\mathrm{int}}\times_{X^l} \left(\prod_{j=1}^l V_j\right).$$
More precisely, $\bM_{1,l}(\mathbf{T},\beta;V_1, \ldots, V_l)$ is the set of all elements
$$([u;p_0;z_1, \ldots, z_l],x_1, \ldots, x_l)\in\bM_{1,l}(\mathbf{T},\beta)\times \prod_{j=1}^l V_j$$
such that $u(z_1, \ldots, z_l)=(x_1,\ldots,x_l)$. The expected dimension of
$\bM_{1,l}(\mathbf{T},\beta;V_1, \ldots, V_l)$ is given by $n+\mu(\beta) + 2l - 2 - \sum_{j=1}^l \textrm{codim}_\bR(V_j)$.

\begin{defn}[\cite{FOOO2,FOOO3}] \label{def_openGW}
The genus zero open Gromov-Witten invariant $n(\beta;V_1, \ldots, V_l)$ is defined as
$$n(\beta;V_1, \ldots, V_l)=\mathrm{ev}_{0*}([\bM_{1,l}(\mathbf{T},\beta;V_1, \ldots, V_l)]^{\virt})\in\rat.$$
It is non-zero only when
$$\mu(\beta) = 2 - 2l + \sum_{j=1}^l \textrm{codim}_\bR(V_j).$$
\end{defn}
By Lemma 6.8 in \cite{FOOO2}, the number $n(\beta;V_1, \ldots, V_l)\in\rat$ is independent of the auxiliary $T^n$-equivariant perturbation data used to define $[\bM_{1,1}(\mathbf{T},\beta;V)]^{\virt}$ and hence gives an invariant.  Definition \ref{one-pt_openGW} is the special case when $l = 0$.

Choose an additive basis $\{T_i=\textrm{PD}[V_i]\}$ of $H^*(X,\bC)$ represented by the Poincar\'{e} duals of fundamental classes of toric subvarieties $V_i\subset X$.
\begin{defn}[\cite{FOOO2,FOOO3}]\label{def5.1}
Define an additive map $\psi:QH^*(X)\rightarrow Jac(W)$ by setting
$$\psi(T_i)=\sum_{\beta:\mu(\beta)=\textrm{codim}_\bR(V_i)}n(\beta;V_i)Z_\beta,$$
and extending linearly.
\end{defn}

\begin{remark}
Fukaya, Oh, Ohta and Ono \cite{FOOO2} also study the so-called potential function with bulk of a toric manifold $X$, by incorporating deformations of Floer cohomology by cycles on the ambient space $X$. (In contrast, the superpotential, or what Fukaya, Oh, Ohta and Ono called the potential function, $W$ just encodes deformations of Floer cohomology by the cycles on $L$.) In the recent work \cite{FOOO4}, they proved that the Jacobian ring of the potential function with bulk is canonically isomorphic to the \textit{big} quantum cohomology ring of $X$. The map $\psi:QH^*(X)\rightarrow Jac(W)$ we discuss here is a special case of this isomorphism, when the bulk deformation is set to zero.  We will also discuss the potential function with bulk in Section \ref{comments}.
\end{remark}

Now, for the toric prime divisors $D_1,\ldots,D_d$, the map $\psi$ is given by
$$D_i\mapsto\sum_{\beta:\mu(\beta)=2}n(\beta;D_i)Z_\beta.$$
A special case of Lemma 9.2 in \cite{FOOO2} gives the following analogue of the divisor equation for open Gromov-Witten invariants.

\begin{prop}[\cite{FOOO2}]\label{divisor}
If $D$ is a toric divisor, then we have the following equality
$$n(\beta;D)=(D\cdot\beta)n_\beta.$$
\end{prop}

Combining with our Theorem \ref{thmGW}, we can compute the map $\psi:QH^*(X)\rightarrow Jac(W)$ on toric divisors for any compact semi-Fano toric surface. As an application, we outline a proof of Corollary \ref{thmQC=JAC} in the following.

To begin with, recall that the cohomology ring $H^*(X,\bC)$ of a compact toric manifold $X$ is generated by the divisor classes $D_1,\ldots,D_d\in H^2(X,\bC)$. Moreover, a presentation of $H^*(X,\bC)$ is given by
$$H^*(X,\bC)=\bC[D_1,\ldots,D_d]/(\mathcal{L}+\mathcal{SR}),$$
where $\mathcal{L}$ is the ideal generated by linear equivalences among divisors and $\mathcal{SR}$ is the Stanley-Reisner ideal generated by primitive relations.

By a result of Siebert and Tian \cite{ST}, when $X$ is semi-Fano, the small quantum cohomology $QH^*(X)$ is also generated by the divisor classes $D_1,\ldots,D_d$ and a presentation of $QH^*(X)$ is given by replacing each relation in $\mathcal{SR}$ by its quantum counterpart, i.e. denoting the quantum Stanley-Reisner ideal by $\mathcal{SR}_Q$, then we have
$$QH^*(X)=\bC[D_1,\ldots,D_d]/(\mathcal{L}+\mathcal{SR}_Q).$$

Consider the case when $X=X_\Sigma$ is a semi-Fani toric surface. We also assume that $X$ is not $\proj^2$. Then any primitive collection is of the form $\mathfrak{P}=\{v_i,v_j\}$ so that $v_i,v_j$ do not generate a cone in $\Sigma$. To compute $\mathcal{SR}_Q$, we need to calculate $D_i\ast D_j$, where $\ast$ denotes the small quantum product. Choose dual bases $\{D_m\}$, $\{D^m\}$ of $H^2(X)$, both represented by toric divisors. Then, by the divisor equation and a straightforward manipulation, we have
\begin{eqnarray*}
D_i\ast D_j & = & \sum_{\alpha:c_1(\alpha)=2}(D_i\cdot\alpha)(D_j\cdot\alpha)GW^{X,\alpha}_{0,1}([\pt])q^\alpha\\
            &   & +\sum_m\left(\sum_{\alpha:c_1(\alpha)=1}(D_i\cdot\alpha)(D_j\cdot\alpha)(D^m\cdot\alpha)
                   GW^{X,\alpha}_{0,0}q^\alpha\right)D_m.
\end{eqnarray*}

The Gromov-Witten invariants $GW^{X,\alpha}_{0,1}([\pt]),GW^{X,\alpha}_{0,0}$ can be computed using the results of Bryan-Leung \cite{BL} as follows. To compute $GW^{X,\alpha}_{0,1}([\pt])$, note that we have $c_1(\alpha)=2$ so that $\alpha^2=0$. Such an $\alpha$ must be of the form $\alpha'+f$ where $\alpha'$ is represented by a chain of $(-2)$-toric prime divisors and $f$ is a fiber class. We are therefore in exactly the same situation as in the proof of Theorem \ref{thmGW}. Hence, $GW^{X,\alpha}_{0,1}([\pt])$ can be computed as before.

As for $GW^{X,\alpha}_{0,0}$, we have $c_1(\alpha)=1$, so that $\alpha$ is represented by a chain $\sum_{k=-p}^qs_kD_{i_k}$ of toric prime divisors such that $D_{i_k}^2=-2$ for all $k\neq0$, $D_{i_0}^2=-1$ and $s_0=1$. The results of Bryan and Leung also apply in this situation: namely, the Gromov-Witten invariant $GW^{X,\alpha}_{0,0}=1$ if both the chains $\sum_{k=-p}^0 s_kD_{i_k}$ and $\sum_{k=0}^q s_kD_{i_k}$ are admissible with center 0 and $GW^{X,\alpha}_{0,0}=0$ otherwise.

Let us give an example to illustrate the explicit computations.\\

\noindent\textbf{Example.} Let $\Sigma$ be the fan whose rays are generated by
$$v_1=(1,0),v_2=(0,1),v_3=(-1,-1),v_4=(0,-1),v_5=(1,-1),v_6=(2,-1).$$
This determines a toric surface $X$. We equip $X$ with a toric K\"ahler form such that the polytope $P$ is given by
\begin{eqnarray*}
P & = & \{(x_1,x_2)\in\bR^2:x_1\geq0,0\leq x_2\leq t_1+t_3+2t_4,x_1+x_2\leq t_1+t_2+2t_3+3t_4,\\
  &   & \qquad\qquad t_1+t_4+x_1-x_2\geq0,t_1+2x_1-x_2\geq0\},
\end{eqnarray*}
where $t_i>0$ are the K\"ahler parameters.

\begin{figure}[ht]
\begin{center}
\includegraphics{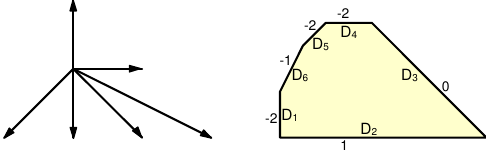}
\end{center}
\caption{The fan $\Sigma$ and the polytope $P$ defining $X$. The numbers beside the divisors indicate their self-intersection numbers.}\label{fig1}
\end{figure}

The linear equivalences among divisors are generated by the following two relations
\begin{eqnarray*}
D_1-D_3+D_5+2D_6=0,\\
D_2-D_3-D_4-D_5-D_6=0.
\end{eqnarray*}
Hence, $H^2(X)$ is of rank 4. We choose the dual bases $\{D^m\}$ and $\{D_m\}$ to be $\{D_1,D_4,D_5,D_6\}$ and $\{D_2,D_3,D_4+2D_3,D_1+2D_2\}$ respectively.

We can now start to compute the primitive relations. For example, we want to compute $D_2\ast D_4$. We need to look for all curve classes with $c_1=1,2$ which intersect both $D_2$ and $D_4$ non-trivially. There are two such classes with $c_1=2$: the classes represented by $D_3$ and $D_3+D_4$, and also two with $c_1=1$: the classes represented by $D_1+D_5+D_6$ and $D_1+D_4+D_5+D_6$. Since all these configurations are admissible, the corresponding Gromov-Witten invariants are all equal to one, by the above discussion. Hence, we get
\begin{eqnarray*}
D_2\ast D_4 & = & q_1q_3q_4^2-q_1q_2q_3q_4^2+q_1q_3q_4(-D_2+D_3-(D_4+2D_3)+(D_1+2D_2))\\
&   & \qquad -q_1q_2q_3q_4(-D_2-D_3+(D_1+2D_2))\\
& = & q_1q_3q_4^2-q_1q_2q_3q_4^2+q_1q_3q_4(D_1+D_5+D_6)\\
&   & \qquad -q_1q_2q_3q_4(D_1+D_4+D_5+D_6),
\end{eqnarray*}
where we have used linear equivalences to get the second equality. Similarly, we can compute all other primitive relations.\\

Having computed all the primitive relations, we can go on to show the following
\begin{lem}
The map
$$\psi:\bC[D_1,\ldots,D_d]\to\bC[z_1^{\pm1},z_2^{\pm1}],\ D_i\mapsto\sum_{\beta:\mu(\beta)=2}n(\beta;D_i)Z_\beta$$
defines a ring homomorphism $\psi:QH^*(X)\to Jac(W)$.
\end{lem}
\begin{proof}[Sketch of proof]
First of all, we show that the ideal $\mathcal{L}$ of linear equivalences is mapped to the ideal $\langle\partial_1W,\ldots,\partial_nW\rangle$ by $\psi$. Linear equivalences are generated by the relations $\sum_{i=1}^d v_i^jD_i=0$, $j=1,2$, where we write $v_i=(v_i^1,v_i^2)$ in coordinates. By Proposition \ref{divisor}, we have
\begin{eqnarray*}
\psi(D_i) & = & \sum_{k=1}^d\sum_{\alpha:c_1(\alpha)=0}n(\beta_k+\alpha;D_i)Z_{\beta_k+\alpha}\\
& = & \sum_{k=1}^d\sum_{\alpha:c_1(\alpha)=0}(D_i\cdot(\beta_k+\alpha))n(\beta_k+\alpha)Z_{\beta_k+\alpha}.
\end{eqnarray*}
Hence, we have
\begin{eqnarray*}
\psi\left(\sum_{i=1}^d v_i^jD_i\right) & = & \sum_{i=1}^d v_i^j\left(\sum_{k=1}^d\sum_{\alpha:c_1(\alpha)=0}(D_i\cdot(\beta_k+\alpha))n(\beta_k+\alpha)
Z_{\beta_k+\alpha}\right)\\
& = & \sum_{k=1}^d\sum_{\alpha:c_1(\alpha)=0}\left(\sum_{i=1}^dv_i^j(\delta_{ik}+D_i\cdot\alpha)\right)
n(\beta_k+\alpha;D_i)Z_{\beta_k+\alpha}\\
& = & \sum_{k=1}^d\sum_{\alpha:c_1(\alpha)=0}v_k^jn(\beta_k+\alpha;D_i)Z_{\beta_k+\alpha}\\
& = & \partial_j W.
\end{eqnarray*}

Next, we need to show that each primitive relation is mapped by $\psi$ to a relation in the ideal $\langle\partial_1W,\ldots,\partial_nW\rangle$. This can be done by explicit computations. Again, we illustrate this by an example.

Consider $X$ in the previous example. By Theorem \ref{thmGW}, we can compute the superpotential explicitly. The result is given by
\begin{eqnarray*}
W & = & (1+q_1)z_1+z_2+\frac{q_1q_2q_3^2q_4^3}{z_1z_2}+(1+q_2+q_2q_3)\frac{q_1q_3q_4^2}{z_2}\\
  &   & \qquad+(1+q_3+q_2q_3)\frac{q_1q_4z_1}{z_2}+\frac{q_1z_1^2}{z_2},
\end{eqnarray*}
where $q_l=\exp(-t_l)$, $l=1,\ldots,4$. We can also compute the images of the divisors $D_i$ under $\psi$:
\begin{eqnarray*}
\psi(D_1) & = & (1-q_1)z_1,\\
\psi(D_2) & = & z_2+q_1z_1,\\
\psi(D_3) & = & \frac{q_1q_2q_3^2q_4^3}{z_1z_2}+(q_2+q_2q_3)\frac{q_1q_3q_4^2}{z_2}+\frac{q_1q_2q_3q_4z_1}{z_2},\\
\psi(D_4) & = & (1-q_2)(\frac{q_1q_3q_4^2}{z_2}+\frac{q_1q_3q_4z_1}{z_2}),\\
\psi(D_5) & = & (1-q_3)(\frac{q_1q_4z_1}{z_2}+\frac{q_1q_2q_3q_4^2}{z_2}),\\
\psi(D_6) & = & \frac{q_1z_1^2}{z_2}+q_1z_1+(q_3+q_2q_3)\frac{q_1q_4z_1}{z_2}+\frac{q_1q_2q_3^2q_4^2}{z_2}.
\end{eqnarray*}
Using what we have computed before,
\begin{eqnarray*}
D_2\ast D_4 & = & q_1q_3q_4^2-q_1q_2q_3q_4^2+q_1q_3q_4(D_1+D_5+D_6)\\
&   & \qquad -q_1q_2q_3q_4(D_1+D_4+D_5+D_6)\\
& = & q_1q_3q_4[(1-q_2)(q_4+D_1+D_5+D_6)-q_2D_4].
\end{eqnarray*}
This is mapped by $\psi$ to
\begin{eqnarray*}
&   & q_1q_3q_4[(1-q_2)(q_4+z_1+\frac{q_1z_1^2}{z_2}+(1+q_2q_3)\frac{q_1q_4z_1}{z_2}+\frac{q_1q_2q_3q_4^2}{z_2})\\
&   & \qquad-q_2(1-q_2)(\frac{q_1q_3q_4^2}{z_2}+q_3\frac{q_1q_4z_1}{z_2})]\\
& = & q_1q_3q_4(1-q_2)(q_4+z_1+\frac{q_1z_1^2}{z_2}+\frac{q_1q_4z_1}{z_2}),
\end{eqnarray*}
which is exactly $\psi(D_2)\cdot\psi(D_4)$.

Similarly, we can show that $\psi(\mathcal{SR}_Q)=\{0\}\subset Jac(W)$. Hence, $\psi$ defines a ring homomorphism $\psi:QH^*(X)\to Jac(W)$.
\end{proof}

Corollary \ref{thmQC=JAC} now follows from the following lemma.
\begin{lem}
For generic choices of the K\"ahler parameters $q_l$, $\psi:QH^*(X)\to Jac(W)$ is a bijective map.
\end{lem}
\begin{proof}[Sketch of proof]
Having computed the superpotential $W$ and the images of the divisors $D_i$ under $\psi$, we can check surjectivity of $\psi$ in a straightforward way. For instance, for the surface $X$ in the previous example, we have
\begin{eqnarray*}
z_1=\psi((1-q_1)^{-1}D_1), z_2=\psi(D_2-q_1(1-q_1)^{-1}D_1),\\
z_2^{-1}=\psi([q_1q_3q_4^2(1-q_2)(1-q_2q_3)]^{-1}D_4-[q_1q_4^2(1-q_3)(1-q_2q_3)]^{-1}D_5).
\end{eqnarray*}
Also, since we have the relation $\partial_1W=0$ which gives
$$z_1^{-1}=(q_1q_2q_3^2q_4^3)^{-1}[(1+q_1)z_1z_2+(1+q_3+q_2q_3)q_1q_4z_1+2q_1z_1^2],$$
and $\psi$ is a homomorphism, $z_1^{-1}$ also lies in the image of $\psi$. The surjectivity of $\psi$ for all other examples can be checked in this way.

On the other hand, by Proposition 3.7 and Lemma 3.9 in Iritani \cite{I} (which were proved by using Kouchnirenko's results), we have $\dim H^*(X)=\dim Jac(W)$ for generic choices of the K\"ahler parameters $q_l$. Hence, $\psi:QH^*(X)\to Jac(W)$ is bijective.
\end{proof}

\section{The big quantum cohomology}\label{comments}

\subsection{The potential with bulk}
For a Lagrangian torus fiber $\mathbf{T}$ in a compact toric manifold $X$ and $\mathbf{b} \in \mathscr{A}$, where $\mathscr{A} := \cpx \langle \textrm{toric invariant cycles} \rangle$, Fukaya, Oh, Ohta and Ono \cite{FOOO2} defined the {\em potential with bulk} $W_{\mathbf{b}}$ as
$$ W_{\mathbf{b}} := \sum_{\substack{\beta \in \pi_2 (X, \mathbf{T})\\ l \geq 0}} \frac{1}{l!} n_l (\beta; \underbrace{\mathbf{b}, \ldots, \mathbf{b}}_l) Z_\beta$$
where the open Gromov-Witten invariants $n(\beta; V_1, \ldots, V_l)$  (see Definition \ref{def_openGW}) extend multilinearly to give a function $n_l: \pi_2 (X, \mathbf{T}) \times \mathscr{A}^{\otimes l} \to \cpx$.  In a recent preprint \cite{FOOO4} they proved that
$$ QH^*_{\mathbf{b}} (X) \cong \mathrm{Jac}(W_{\mathbf{b}}).$$
Thus an explicit expression of $W_{\mathbf{b}}$ would give an explicit presentation of the big quantum cohomology ring $QH^*_{\mathbf{b}} (X)$.

In the previous section, we have given an explicit expression of $W_{\mathbf{b}}$ when $\mathbf{b} = 0$ for a semi-Fano toric surface $X$.  We consider its potential with bulk in this section.  For the purpose of computing $QH^*_{\mathbf{b}} (X)$, it is enough to consider $\mathbf{b} = a X + D + c p$, where $D$ is a toric divisor, $p$ is the intersection point of two toric prime divisors (say $D_1$ and $D_2$), and $a, c \in \cpx$.

\begin{prop} [Restatement of Corollary \ref{bulk_cor}] \label{bulk}
Let $X$ be a semi-Fano toric surface, and $\mathbf{b} = a X + D + c p$ as described above.  Then
$$ W_{\mathbf{b}} = a + \sum_{\beta \not= 0} \exp (\langle \beta, D \rangle) \left(\sum_{k=0}^\infty \frac{c^k}{k!} n_k (\beta; p, \ldots, p) \right) Z_\beta.$$
In particular, when $c = 0$,
$$W_{\mathbf{b}} = a + \sum_{\beta \textrm{ admissible}} \exp(\langle \beta, D \rangle) Z_{\beta}.$$
\end{prop}

\begin{proof}
When $\beta \not= 0$,
$$n_k (\beta; [X], \gamma_1, \ldots, \gamma_{k-1}) = 0$$
for all $k \geq 1$ and $\gamma_1, \ldots, \gamma_{k-1} \in H_* (X)$ due to dimension reason.  Thus
\begin{align*}
W_{\mathbf{b}} &:= \sum_{\substack{\beta \in \pi_2 (X, \mathbf{T})\\ l \geq 0}} \frac{1}{l!} n_l (\beta; \mathbf{b}, \ldots, \mathbf{b}) Z_\beta \\
&= \sum_{l \geq 0} \frac{1}{l!} n_l (0; \mathbf{b}, \ldots, \mathbf{b}) + \sum_{\substack{\beta \not= 0\\ l \geq 0}} \frac{1}{l!} n_l (\beta; D + c p, \ldots, D + c p) Z_\beta.
\end{align*}
Moreover, $n_1 (0; X) = 1$ ($\bM_{1,1}(\mathbf{T},0;X)$ contains the constant map only) and $n_1 (0; p) = n_1 (0; D) = 0$ (the corresponding moduli spaces are empty).  Also by dimension counting, $n_l (0; \gamma_1, \ldots, \gamma_l) = 0$ for all $l \not= 1$.  Thus the first term is
$$\sum_{l \geq 0} \frac{1}{l!} n_l (0; \mathbf{b}, \ldots, \mathbf{b}) = a.$$
Using the divisor equation for open Gromov-Witten invariants (\cite{FOOO2}; see Proposition \ref{divisor}), the second term is
\begin{align*}
\sum_{\substack{\beta \not= 0\\ l \geq 0}} \frac{1}{l!} n_l (\beta; D + c p, \ldots, D + c p) Z_\beta &= \sum_{\substack{\beta \not= 0\\ l \geq 0}} \frac{1}{l!} \sum_{k=0}^l \mathrm{C}^l_k c^k n_l (\beta; \underbrace{D, \ldots, D}_{l-k}, \underbrace{p, \ldots, p}_k) Z_\beta \\
&= \sum_{\substack{\beta \not= 0\\ l \geq 0}} \frac{1}{l!} \sum_{k=0}^l \mathrm{C}^l_k c^k (\langle \beta, D \rangle)^{l-k} n_k (\beta;p, \ldots, p) Z_\beta \\
&= \sum_{\substack{\beta \not= 0\\ j,k \geq 0}} \frac{c^k}{j!k!} (\langle \beta, D \rangle)^{j} n_k (\beta;p, \ldots, p) Z_\beta \\
&= \sum_{\beta \not= 0} \exp (\langle \beta, D \rangle) \left(\sum_{k=0}^\infty \frac{c^k}{k!} n_k (\beta; p, \ldots, p) \right) Z_\beta.
\end{align*}
When $c = 0$,
$$W_{\mathbf{b}} = a + \sum_{\beta \not= 0} \exp (\langle \beta, D \rangle)  n_\beta  Z_\beta.$$
By Theorem \ref{thmGW}, $n_\beta = 1$ when $\beta$ is admissible, and $0$ otherwise.  Thus
$$W_{\mathbf{b}} = a + \sum_{\beta \textrm{ admissible}} \exp(\langle \beta, D \rangle) Z_{\beta}.$$
\end{proof}

\subsection{Speculations and discussions}
In Proposition \ref{bulk}, $n_l(\beta; p, \ldots, p)$ $(l \geq 1)$ has not been computed.  In the following we give an informal discussion concerning these invariants.

One of the issues involved in computing these invariants is the presence of ``ghost bubbles" in the moduli space $\bM_{1,l}(\mathbf{T},\beta;p, \ldots, p)$ (see Figure \ref{ghost}) when $p$ is chosen to be a toric fixed point.  On the other hand, if we consider $p_1, \ldots, p_l \in X$ in generic positions, which is the approach taken by Gross \cite{Gross} where he used tropical geometry to define the superpotential with bulk, the moduli space $\bM_{1,l}(\mathbf{T},\beta;p_1, \ldots, p_l)$ does not involve disk bubbling (when $\beta$ has the suitable Maslov index
$$\mu(\beta) = 2 - 2l + \sum_{j=1}^l \textrm{codim}_\bR(p_j\subset X) = 2 + 2l$$
so that the moduli has expected dimension $n = \dim \mathbf{T}$)\footnote{In general stable discs with Maslov index greater than or equal to two have disc bubblings.  The moduli space $\bM_{1,l}(\mathbf{T})$ has codimension-one boundary and the invariant depends on the perturbation data.  The work of Fukaya-Oh-Ono-Ohta restricted to $T^n$-equivariant perturbations to define the invariants.  On the other hand, if the disc is required to pass through $l$ points in generic positions ($l$ determined by $\mu(\beta) = 2 + 2l$), then it consists of only one disc component and hence disc bubbling cannot occur.  This means $\bM_{1,l}(\mathbf{T},\beta;p_1, \ldots, p_l)$ does not have codimension-one boundary, and so the corresponding invariant is independent of the perturbation data.}.  Since the moduli space $\bM_{1,l}(\mathbf{T},\beta;p_1, \ldots, p_l)$ does not have codimension-one boundary, the invariant $n_l(\beta; p_1, \ldots, p_l)$ are well-defined.  The invariants may become more computable.

\begin{figure}[h]
\begin{center}
\includegraphics{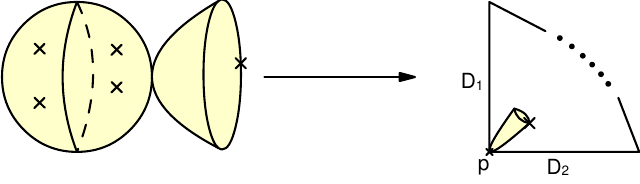}
\end{center}
\caption{Ghost bubbles in $\bM_{1,4}(\mathbf{T},\beta;p, p, p, p)$.  The whole sphere bubble is contracted to the toric fixed point $p$. The disk class is taken such that $\bM_{1,4}(\mathbf{T},\beta;p, p, p, p)$ has expected dimension $2$.  However the actual dimension is bigger than $2$ since the interior marked points are free to move in the bubble.} \label{ghost}
\end{figure}

This motivates us to consider $p' \in D_1$ which is \emph{not} fixed by the torus action, and define the invariant $n_l(\beta; p', \ldots, p')$ by taking a generic perturbation of the $l$ points around $p'$.

\begin{figure}[h]
\begin{center}
\includegraphics{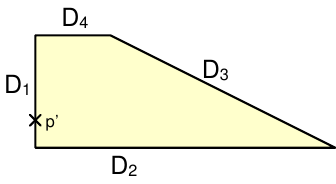}
\end{center}
\caption{The polytope of the Hirzebruch surface $\mathbf{F}_2$.} \label{F_2}
\end{figure}

\begin{eg} [The Hirzebruch surface $\mathbf{F}_2$]
Consider The Hirzebruch surface $\mathbf{F}_2$ whose polytope picture is shown in Figure \ref{F_2}.  If we take the above approach, then
$n_l(\beta; p', \ldots, p')$ equals to $1$ when $\beta = l \beta_1 + \beta_i$ for $i = 2, 3, 4$ or $\beta = l \beta_1 + \beta_4 + D_4$, and $0$ otherwise.  Then for $\mathbf{b} =  a [X] + D + c p$,
\begin{align*}
W_{\mathbf{b}} =& \, a + \sum_{\beta \not= 0} \exp (\langle \beta, D \rangle) \left(\sum_{k=0}^\infty \frac{c^k}{k!} n_k (\beta; p, \ldots, p) \right) Z_\beta \\
=& \, a + \exp (\langle \beta_1, D \rangle) Z_{\beta_1} + \sum_{i=2}^4 \exp (c \, \conste^{\langle \beta_1, D \rangle} Z_{\beta_1}) \exp (\langle \beta_i, D \rangle) Z_{\beta_i} \\
& + \exp (c \, \conste^{\langle \beta_1, D \rangle} Z_{\beta_1}) \exp (\langle \beta_4 + D_4, D \rangle) q_4 Z_{\beta_4}.
\end{align*}
\end{eg}

The above consideration is tentative, and we are still investigating whether this idea is in the right direction.

\appendix
\section{List of superpotentials for semi-Fano toric surfaces}\label{table}
Using the fact that any smooth compact toric surface is a blowup of either $\proj^2$ or a Hirzebruch surface $\bF_m$ ($m\geq0$) at torus fixed points, it is easy to see that there are finitely many isomorphism classes of semi-Fano toric surfaces. In fact, all except $\bF_2$ and $\proj^1\times\proj^1$ are blowups of $\proj^2$; there are 16 of such surfaces, five of which are Fano (namely, $\proj^2$, $\proj^1\times\proj^1$ and the blowup of $\proj^2$ at 1, 2 or 3 points).

By using Theorem \ref{thmGW}, we can compute the superpotentials for all these surfaces explicitly. In this appendix, we provide a list of the superpotentials for the 11 semi-Fano but non-Fano toric surfaces. We enumerate them as $X_1,\ldots,X_{11}$, and each surface is specified by the primitive generators $\rho(\Sigma)$ of rays of its fan and the defining inequalities of its polytope. Also, in the following tables, the $t_l$'s are positive numbers and $q_l=\exp(-t_l)$ are the K\"ahler parameters.\\

\begin{tabular}{|p{0.45cm}|p{2.2cm}|p{3.4cm}|p{5.5cm}|}

\hline

& $\rho(\Sigma)$ & polytope $P$ & superpotential $W$\\

\hline

\multirow{4}{*}{$X_1$}
& $v_1=(1,0)$ & $x_1\geq0$ & \multirow{4}{5.5cm}{$z_1+z_2+\frac{q_1^2q_2}{z_1z_2^2}+(1+q_2)\frac{q_1}{z_2}$}\\
& $v_2=(0,1)$ & $x_2\geq0$ & \\
& $v_3=(-1,-2)$ & $2t_1+t_2-x_1-2x_2\geq0$ & \\
& $v_4=(0,-1)$ & $t_1-x_2\geq0$ & \\

\hline

\multirow{5}{*}{$X_2$}
& $v_1=(1,0)$ & $x_1\geq0$ & \multirow{5}{5.5cm}{$z_1+z_2+\frac{q_1q_2q_3^2}{z_1z_2}+(1+q_2)\frac{q_1q_3}{z_2}+\frac{q_1z_1}{z_2}$}\\
& $v_2=(0,1)$ & $x_2\geq0$ & \\
& $v_3=(-1,-1)$ & $t_1+t_2+2t_3-x_1-x_2\geq0$ & \\
& $v_4=(0,-1)$ & $t_1+t_3-x_2\geq0$ & \\
& $v_5=(1,-1)$ & $t_1+x_1-x_2\geq0$ & \\

\hline

\multirow{6}{*}{$X_3$}
& $v_1=(1,0)$ & $x_1\geq0$ & \multirow{6}{5.5cm}{$(1+q_1)z_1+z_2+\frac{q_1q_2q_3^2q_4^3}{z_1z_2}+(1+q_2+q_2q_3)\frac{q_1q_3q_4^2}{z_2}
+(1+q_3+q_2q_3)\frac{q_1q_4z_1}{z_2}+\frac{q_1z_1^2}{z_2}$}\\
& $v_2=(0,1)$ & $x_2\geq0$ & \\
& $v_3=(-1,-1)$ & $t_1+t_2+2t_3+3t_4-x_1-x_2\geq0$ & \\
& $v_4=(0,-1)$ & $t_1+t_3+2t_4-x_2\geq0$ & \\
& $v_5=(1,-1)$ & $t_1+t_4+x_1-x_2\geq0$ & \\
& $v_6=(2,-1)$ & $t_1+2x_1-x_2\geq0$ & \\

\hline

\multirow{6}{*}{$X_4$}
& $v_1=(1,0)$ & $x_1\geq0$ & \multirow{6}{5.5cm}{$(1+q_1)z_1+z_2+\frac{q_2q_3q_4}{z_1}+\frac{q_1q_3q_4^2}{z_2}+(1+q_3)\frac{q_1q_4z_1}{z_2}
+\frac{q_1z_1^2}{z_2}$}\\
& $v_2=(0,1)$ & $x_2\geq0$ & \\
& $v_3=(-1,0)$ & $t_2+t_3+t_4-x_1\geq0$ & \\
& $v_4=(0,-1)$ & $t_1+t_3+2t_4-x_2\geq0$ & \\
& $v_5=(1,-1)$ & $t_1+t_4+x_1-x_2\geq0$ & \\
& $v_6=(2,-1)$ & $t_1+2x_1-x_2\geq0$ & \\

\hline

\multirow{6}{*}{$X_5$}
& $v_1=(1,0)$ & $x_1\geq0$ & \multirow{6}{5.5cm}{$z_1+z_2+\frac{q_2q_3q_4}{z_1}+\frac{q_1q_3q_4^2}{z_1z_2}+(1+q_3)\frac{q_1q_4}{z_2}
+\frac{q_1z_1}{z_2}$}\\
& $v_2=(0,1)$ & $x_2\geq0$ & \\
& $v_3=(-1,0)$ & $t_2+t_3+t_4-x_1\geq0$ & \\
& $v_4=(-1,-1)$ & $t_1+t_3+2t_4-x_1-x_2\geq0$ & \\
& $v_5=(0,-1)$ & $t_1+t_4-x_2\geq0$ & \\
& $v_6=(1,-1)$ & $t_1+x_1-x_2\geq0$ & \\

\hline

\multirow{7}{*}{$X_6$}
& $v_1=(1,0)$ & $x_1\geq0$ & \multirow{7}{5.5cm}{$(1+q_1)z_1+z_2+\frac{q_2q_3q_4q_5}{z_1}+\frac{q_1q_3q_4^2q_5^3}{z_1z_2}
+(1+q_3+q_3q_4)\frac{q_1q_4q_5^2}{z_2}+(1+q_4+q_3q_4)\frac{q_1q_5z_1}{z_2}+\frac{q_1z_1^2}{z_2}$}\\
& $v_2=(0,1)$ & $x_2\geq0$ & \\
& $v_3=(-1,0)$ & $t_2+t_3+t_4+t_5-x_1\geq0$ & \\
& $v_4=(-1,-1)$ & $t_1+t_3+2t_4+3t_5-x_1-x_2\geq0$ & \\
& $v_5=(0,-1)$ & $t_1+t_4+2t_5-x_2\geq0$ & \\
& $v_6=(1,-1)$ & $t_1+t_5+x_1-x_2\geq0$ & \\
& $v_7=(2,-1)$ & $t_1+2x_1-x_2\geq0$ & \\

\hline

\end{tabular}

\begin{tabular}{|p{0.45cm}|p{2.2cm}|p{3.4cm}|p{5.5cm}|}

\hline

& $\rho(\Sigma)$ & polytope $P$ & superpotential $W$\\

\hline

\multirow{7}{*}{$X_7$}
& $v_1=(1,0)$ & $x_1\geq0$ & \multirow{7}{5.5cm}{$(1+q_1)z_1+z_2+\frac{q_2q_3z_2}{q_1q_5z_1}+\frac{q_3q_4q_5}{z_1}+\frac{q_1q_4q_5^2}{z_2}
+(1+q_4)\frac{q_1q_5z_1}{z_2}+\frac{q_1z_1^2}{z_2}$}\\
& $v_2=(0,1)$ & $x_2\geq0$ & \\
& $v_3=(-1,1)$ & $t_2+t_3-t_1-t_5-x_1+x_2\geq0$ & \\
& $v_4=(-1,0)$ & $t_3+t_4+t_5-x_1\geq0$ & \\
& $v_5=(0,-1)$ & $t_1+t_4+2t_5-x_2\geq0$ & \\
& $v_6=(1,-1)$ & $t_1+t_5+x_1-x_2\geq0$ & \\
& $v_7=(2,-1)$ & $t_1+2x_1-x_2\geq0$ & \\

\hline

\multirow{8}{*}{$X_8$}
& $v_1=(1,0)$ & $x_1\geq0$ & \multirow{8}{5.5cm}{$(1+q_1)z_1+z_2+(1+\frac{q_1q_5q_6^2}{q_2^2q_3})\frac{q_2q_3q_4q_5q_6}{z_1}
+\frac{q_1q_3q_4^2q_5^3q_6^4}{z_1^2z_2}+(1+q_3+q_3q_4+q_3q_4q_5)\frac{q_1q_4q_5^2q_6^3}{z_1z_2}
+(1+q_4+q_3q_4+q_4q_5+q_3q_4q_5+q_3q_4^2q_5)\frac{q_1q_5q_6^2}{z_2}
+(1+q_5+q_4q_5+q_3q_4q_5)\frac{q_1q_6z_1}{z_2}+\frac{q_1z_1^2}{z_2}$}\\
& $v_2=(0,1)$ & $x_2\geq0$ & \\
& $v_3=(-1,0)$ & $t_2+t_3+t_4+t_5+t_6-x_1\geq0$ & \\
& $v_4=(-2,-1)$ & $t_1+t_3+2t_4+3t_5+4t_6-2x_1-x_2\geq0$ & \\
& $v_5=(-1,-1)$ & $t_1+t_4+2t_5+3t_6-x_1-x_2\geq0$ & \\
& $v_6=(0,-1)$ & $t_1+t_5+2t_6-x_2\geq0$ & \\
& $v_7=(1,-1)$ & $t_1+t_6+x_1-x_2\geq0$ & \\
& $v_8=(2,-1)$ & $t_1+2x_1-x_2\geq0$ & \\

\hline

\multirow{8}{*}{$X_9$}
& $v_1=(1,0)$ & $x_1\geq0$ & \multirow{8}{5.5cm}{$(1+q_1)z_1+z_2+\frac{q_2q_3^2q_4z_2}{q_1q_6z_1}+(1+q_2)\frac{q_3q_4q_5q_6}{z_1}
+\frac{q_1q_4q_5^2q_6^3}{z_1z_2}+(1+q_4+q_4q_5)\frac{q_1q_5q_6^2}{z_2}+(1+q_5+q_4q_5)\frac{q_1q_6z_1}{z_2}
+\frac{q_1z_1^2}{z_2}$}\\
& $v_2=(0,1)$ & $x_2\geq0$ & \\
& $v_3=(-1,1)$ & $t_2+2t_3+t_4-t_1-t_6-x_1+x_2\geq0$ & \\
& $v_4=(-1,0)$ & $t_3+t_4+t_5+t_6-x_1\geq0$ & \\
& $v_5=(-1,-1)$ & $t_1+t_4+2t_5+3t_6-x_1-x_2\geq0$ & \\
& $v_6=(0,-1)$ & $t_1+t_5+2t_6-x_2\geq0$ & \\
& $v_7=(1,-1)$ & $t_1+t_6+x_1-x_2\geq0$ & \\
& $v_8=(2,-1)$ & $t_1+2x_1-x_2\geq0$ & \\

\hline

\multirow{8}{*}{$X_{10}$}
& $v_1=(1,0)$ & $x_1\geq0$ & \multirow{8}{5.5cm}{$(1+q_1)z_1+z_2+(1+\frac{q_1q_5q_6^2}{q_2^2q_3})\frac{q_2q_3q_4z_2}{q_1q_6z_1}
+\frac{q_4^2q_5z_2}{q_1q_3z_1^2}+(1+q_3)\frac{q_4q_5q_6}{z_1}+\frac{q_1q_5q_6^2}{z_2}
+(1+q_5)\frac{q_1q_6z_1}{z_2}+\frac{q_1z_1^2}{z_2}$}\\
& $v_2=(0,1)$ & $x_2\geq0$ & \\
& $v_3=(-1,1)$ & $t_2+t_3+t_4-t_1-t_6-x_1+x_2\geq0$ & \\
& $v_4=(-2,1)$ & $2t_4+t_5-t_1-t_3-2x_1+x_2\geq0$ & \\
& $v_5=(-1,0)$ & $t_4+t_5+t_6-x_1\geq0$ & \\
& $v_6=(0,-1)$ & $t_1+t_5+2t_6-x_2\geq0$ & \\
& $v_7=(1,-1)$ & $t_1+t_6+x_1-x_2\geq0$ & \\
& $v_8=(2,-1)$ & $t_1+2x_1-x_2\geq0$ & \\

\hline

\multirow{9}{*}{$X_{11}$}
& $v_1=(1,0)$ & $x_1\geq0$ & \multirow{9}{5.5cm}{$(1+q_1+\frac{q_2q_3^2q_4^3q_5}{q_1q_6q_7^3})z_1+(1+\frac{q_2q_3^2q_4^3q_5}{q_1^2q_6q_7^3}
+\frac{q_2q_3^2q_4^3q_5}{q_1q_6q_7^3})z_2+\frac{q_2q_3^2q_4^3q_5z_2^2}{q_1^2q_6q_7z_1}
+(1+q_2+q_2q_3)\frac{q_3q_4^2q_5z_2}{q_1q_7z_1}+(1+q_3+q_2q_3)\frac{q_4q_5q_6q_7}{z_1}+\frac{q_1q_5q_6^2q_7^3}{z_1z_2}
+(1+q_5+q_5q_6)\frac{q_1q_6q_7^2}{z_2}+(1+q_6+q_5q_6)\frac{q_1q_7z_1}{z_2}+\frac{q_1z_1^2}{z_2}$}\\
& $v_2=(0,1)$ & $x_2\geq0$ & \\
& $v_3=(-1,2)$ & $t_2+2t_3+3t_4+t_5-2t_1-t_6-3t_7-x_1+2x_2\geq0$ & \\
& $v_4=(-1,1)$ & $t_3+2t_4+t_5-t_1-t_7-x_1+x_2\geq0$ & \\
& $v_5=(-1,0)$ & $t_4+t_5+t_6+t_7-x_1\geq0$ & \\
& $v_6=(-1,-1)$ & $t_1+t_5+2t_6+3t_7-x_1-x_2\geq0$ & \\
& $v_7=(0,-1)$ & $t_1+t_6+2t_7-x_2\geq0$ & \\
& $v_8=(1,-1)$ & $t_1+t_7+x_1-x_2\geq0$ & \\
& $v_9=(2,-1)$ & $t_1+2x_1-x_2\geq0$ & \\

\hline

\end{tabular}

\newpage

\begin{figure}[ht]
\begin{center}
\includegraphics{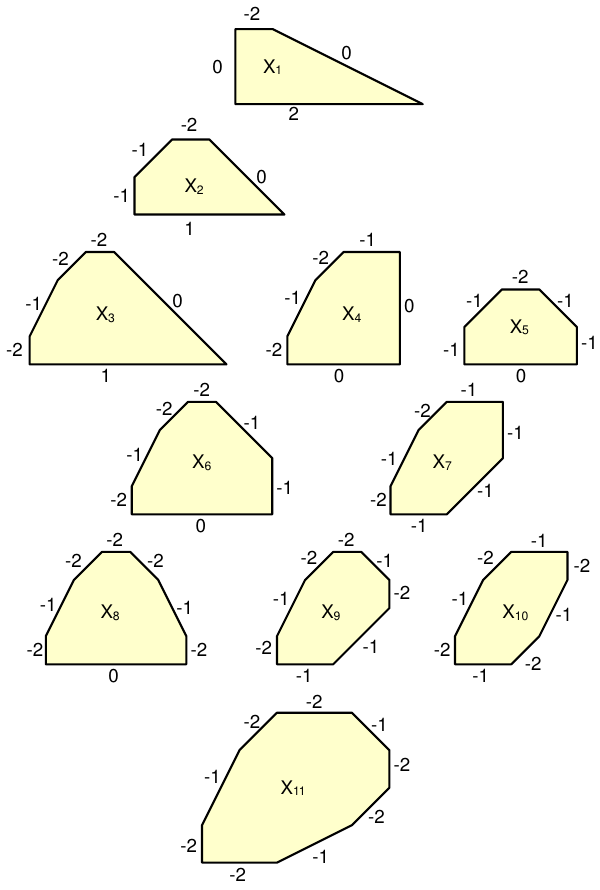}
\end{center}
\caption{Polytopes defining the semi-Fano but non-Fano toric surfaces. The numbers indicate the self-intersection numbers of the toric divisors.}
\end{figure}

\newpage

\bibliographystyle{amsplain}
\bibliography{geometry}

\providecommand{\bysame}{\leavevmode\hbox to3em{\hrulefill}\thinspace}
\providecommand{\MR}{\relax\ifhmode\unskip\space\fi MR }
\providecommand{\MRhref}[2]{%
  \href{http://www.ams.org/mathscinet-getitem?mr=#1}{#2}
}
\providecommand{\href}[2]{#2}
\begin{thebibliography}{10}

\bibitem{A}
D.~Auroux, \emph{Mirror symmetry and {$T$}-duality in the complement of an
  anticanonical divisor}, J. G\"okova Geom. Topol. GGT \textbf{1} (2007),
  51--91. \MR{2386535 (2009f:53141)}

\bibitem{A2}
\bysame, \emph{Special {L}agrangian fibrations, wall-crossing, and mirror
  symmetry}, Surveys in differential geometry. {V}ol. {XIII}. {G}eometry,
  analysis, and algebraic geometry: forty years of the {J}ournal of
  {D}ifferential {G}eometry, Surv. Differ. Geom., vol.~13, Int. Press,
  Somerville, MA, 2009, pp.~1--47. \MR{2537081 (2010j:53181)}

\bibitem{BL}
J.~Bryan and N.C. Leung, \emph{The enumerative geometry of {$K3$} surfaces and
  modular forms}, J. Amer. Math. Soc. \textbf{13} (2000), no.~2, 371--410
  (electronic). \MR{1750955 (2001i:14071)}

\bibitem{C}
K.~Chan, \emph{A formula equating open and closed {G}romov-{W}itten invariants
  and its applications to mirror symmetry}, Pacific J. Math. \textbf{254}
  (2011), no.~2, 275--293. \MR{2900016}

\bibitem{CL}
K.~Chan and N.C. Leung, \emph{Mirror symmetry for toric {F}ano manifolds via
  {SYZ} transformations}, Adv. Math. \textbf{223} (2010), no.~3, 797--839.
  \MR{2565550 (2011k:14047)}

\bibitem{CO}
C.-H. Cho and Y.-G. Oh, \emph{Floer cohomology and disc instantons of
  {L}agrangian torus fibers in {F}ano toric manifolds}, Asian J. Math.
  \textbf{10} (2006), no.~4, 773--814. \MR{2282365 (2007k:53150)}

\bibitem{FOOO4}
K.~Fukaya, Y.-G. Oh, H.~Ohta, and K.~Ono, \emph{Lagrangian {F}loer theory and
  mirror symmetry on compact toric manifolds}, preprint,
  \href{http://arxiv.org/abs/1009.1648}{arXiv:1009.1648}.

\bibitem{FOOO5}
\bysame, \emph{Spectral invariants with bulk, quasimorphisms and {L}agrangian
  {F}loer theory}, preprint,
  \href{http://arxiv.org/abs/1105.5123}{arXiv:1105.5123}.

\bibitem{FOOO}
\bysame, \emph{Lagrangian {F}loer theory on compact toric manifolds. {I}}, Duke
  Math. J. \textbf{151} (2010), no.~1, 23--174. \MR{2573826 (2011d:53220)}

\bibitem{FOOO2}
\bysame, \emph{Lagrangian {F}loer theory on compact toric manifolds {II}: bulk
  deformations}, Selecta Math. (N.S.) \textbf{17} (2011), no.~3, 609--711.
  \MR{2827178}

\bibitem{FOOO3}
\bysame, \emph{Toric degeneration and nondisplaceable {L}agrangian tori in
  {$S^2\times S^2$}}, Int. Math. Res. Not. IMRN (2012), no.~13, 2942--2993.
  \MR{2946229}

\bibitem{G}
A.~Gathmann, \emph{Gromov-{W}itten invariants of blow-ups}, J. Algebraic Geom.
  \textbf{10} (2001), no.~3, 399--432. \MR{1832328 (2002b:14069)}

\bibitem{Gross}
M.~Gross, \emph{Mirror symmetry for {$\Bbb P^2$} and tropical geometry}, Adv.
  Math. \textbf{224} (2010), no.~1, 169--245. \MR{2600995 (2011j:14089)}

\bibitem{HV}
K.~Hori and C.~Vafa, \emph{Mirror symmetry}, preprint,
  \href{http://arxiv.org/abs/hep-th/0002222}{arXiv:hep-th/0002222}.

\bibitem{H}
J.~Hu, \emph{Gromov-{W}itten invariants of blow-ups along points and curves},
  Math. Z. \textbf{233} (2000), no.~4, 709--739. \MR{1759269 (2001c:53115)}

\bibitem{I}
H.~Iritani, \emph{An integral structure in quantum cohomology and mirror
  symmetry for toric orbifolds}, Adv. Math. \textbf{222} (2009), no.~3,
  1016--1079. \MR{2553377 (2010j:53182)}

\bibitem{LLW}
S.-C. Lau, N.C. Leung, and B.~Wu, \emph{A relation for {G}romov-{W}itten
  invariants of local {C}alabi-{Y}au threefolds}, Math. Res. Lett. \textbf{18}
  (2011), no.~5, 943--956. \MR{2875867}

\bibitem{LLW2}
\bysame, \emph{Mirror maps equal {SYZ} maps for toric {C}alabi-{Y}au surfaces},
  Bull. Lond. Math. Soc. \textbf{44} (2012), no.~2, 255--270. \MR{2914605}

\bibitem{L}
N.C. Leung, \emph{Mirror symmetry without corrections}, Comm. Anal. Geom.
  \textbf{13} (2005), no.~2, 287--331. \MR{2154821 (2006c:32028)}

\bibitem{ST}
B.~Siebert and G.~Tian, \emph{On quantum cohomology rings of {F}ano manifolds
  and a formula of {V}afa and {I}ntriligator}, Asian J. Math. \textbf{1}
  (1997), no.~4, 679--695. \MR{1621570 (99d:14060)}

\bibitem{SYZ}
A.~Strominger, S.-T. Yau, and E.~Zaslow, \emph{Mirror symmetry is
  {$T$}-duality}, Nuclear Phys. B \textbf{479} (1996), no.~1-2, 243--259.
  \MR{1429831 (97j:32022)}

\end{thebibliography}

\end{document}